\documentclass[10pt]{amsart}
\usepackage{enumerate,url,amssymb,  mathrsfs, nicefrac, pifont, stmaryrd, graphicx, pdfsync}
\newtheorem{theorem}{Theorem}[section]

\newtheorem*{lemma*}{Lemma}

\theoremstyle{definition}

\theoremstyle{remark}
\newtheorem{remark}[theorem]{Remark}

\numberwithin{equation}{section}

\newcommand{\onto}{\xrightarrow[]{{}_{\!\!\textnormal{onto}\!\!}}}

\newcommand{\deff}{\stackrel {\tiny{\textnormal{def}}}{=\!\!=} }

\def\Xint#1{\mathchoice
{\XXint\displaystyle\textstyle{#1}}%
{\XXint\textstyle\scriptstyle{#1}}%
{\XXint\scriptstyle\scriptscriptstyle{#1}}%
{\XXint\scriptscriptstyle\scriptscriptstyle{#1}}%
\!\int}
\def\XXint#1#2#3{{\setbox0=\hbox{$#1{#2#3}{\int}$}\vcenter{\hbox{$#2#3$}}\kern-.5\wd0}}
\def\dashint{\Xint-}

\def\XXiint#1#2#3{{\setbox0=\hbox{$#1{#2#3}{\iint}$}\vcenter{\hbox{$#2#3$}}\kern-.5\wd0}}

\begin{document}

\title[The Beauty of Harmonic Analysis]  {Let the Beauty of Harmonic Analysis  \\ be Revealed Through  Nonlinear PDE{\footnotesize{s}} \\$\;$ \\ $\textnormal{\textit{{\large A Work of Art in Three  Sketches}}}$}

\author[Iwaniec]{Tadeusz \;\;Iwaniec  }
\address{Department of Mathematics, Syracuse University, Syracuse,
NY 13244, USA and Department of Mathematics and Statistics,
University of Helsinki, Finland}
\email{tiwaniec@syr.edu}

\thanks{ T. Iwaniec was supported by the NSF grant DMS-0800416 and the Academy of Finland project 1128331.}




\maketitle




\section*{Prologue}
Mathematicians, like me, have the privilege to enjoy the ingenious ideas and splendid theories imagined and brilliantly developed by previous mathematicians. Their beauty inspires us to ask questions to create our own little theory with grace and prospective applications.  And there is never an end to new questions, which in effect is the key to advances in mathematics.  But advances come after hours and hours of intense work, trapping and holding our attention for years.  This can be a dream, sometimes immense pleasure, sometimes a breathtaking moment when we spot the underlying ideas  that are actually relevant to our aspirations. Genuine mathematics does not abide in complexity but, contrary to what one might think, somewhere in the unlimited beauty of applications of sophisticated ideas. Paraphrasing Luciano Pavarotti on music, let me say:
\begin{quote}
\textit{``Learning mathematics  by only reading about it  is like making love by e-mail.''}
\end{quote}

From the very beginning of my mathematical life I fell in love with logic and later as a young scholar with geometry and harmonic analysis. There are so many captivating topics in geometric analysis. I was especially fascinated by  the foundation of Geometric Function Theory (GFT, quasiconformal mappings), the mysteries in the Calculus of Variations (nonconvex energy integrals) and  Nonlinear Partial Differential Equations (PDEs, elliptic type). Nowadays,  these fields are essential in  material science and nonlinear elasticity, which are critical in modern technology and many engineering problems. Myriad practical problems of nonlinear elasticity and numerous elegant conjectures are very appealing to me.  But I cannot fully treat these topics here. I will only indicate briefly a few adventurous moments of my studies on these topics by means of applications of
maximal functions due to \textit{Hardy}\, and  \textit{Littlewood}, \textit{Fefferman} and \textit{Stein}, as well as nonlinear commutators which originated with \textit{Coifman}, \textit{Rochberg} and \textit{Weiss}. The results presented here are not my best, though there is some element of aesthetic beauty in them. It is for these reasons that:
\begin{center}
 \textit{\textbf{Fefferman} and \textbf{Stein} (the architects of maximal inequalities),\\  \textit{\textbf{Coifman}}, \textit{\textbf{Rochberg}} and \textit{\textbf{Weiss}} (the founders of singular commutators)\\ became  my mathematical luminaries. }
\end{center}

  Some further results (on quasiconformal mappings in even dimensions) are discussed in  this issue by my amazingly imaginative colleague \textit{Gaven Martin}\footnote{Gaven speaks of me as "After all, he has quite a good memory even if it is a bit short" Yes indeed, I have a good memory for masterpieces, but a short one for trivia.}. We have presented GFT in all dimensions in our book \cite{IMb}.  Well, we did not make a fortune with this book, nor did we become famous. But I have heard someone say, "Hey, I have read your book".  How  satisfying!\\

As I share the beauty and joy of mathematics with you I also remember Polish mathematicians whose glorious scientific careers came to a cruel end during Nazi-Soviet occupation. \textit{J\'{o}zef Marcinkiewicz}, \textit{Stanis{\l}aw Saks} and \textit{Juliusz Pawe{\l} Schauder} were inspirations to me. I am mindful of them not only as mathematicians  \cite{I}. Marcinkiewicz, along with  22 thousand Polish patriots who dared to exhibit a love and pride of an independent Poland, were executed by the order of J. Stalin, and buried secretly in mass graves in gloomy forested sites near Starobielsk, Ostashkovo and the most documented Katy\'{n}.

\begin{quote}
 \textit{KATY\'{N}  CAROL}

\textit{Someday maybe a great musician will rise up,\\
will transform speechless rows of gravestones into a keyboard,\\
a great Polish song writer will compose a frightening ballad with blood and tears.\\
$[...]$\\
And there will emerge untold stories,\\
strange hearts, bodies bathed in light...\\
And the Truth again will embody\\
The Spirit\\
with living words-of the sand of Katy\'{n}} \\

\end{quote}
\begin{flushright} - \textit{Kazimiera I{\l}{\l}akowicz\'{o}wna} \\
(translated by the author of this article)
\end{flushright}
 Antoni Zygmund remarked once about Marcinkiewicz:
\begin{quote}
``...his early death may be seen as a great blow to Polish Mathematics, and probably its heaviest individual loss during the Second World War.''\\
\end{quote}

I have had the privilege of growing up in the environment these mathematicians left for us.

\section{The Natural Domain of Definition}
While singular integrals are naturally defined in $\,\mathscr L^2(\mathbb R^n) $, there is also such a thing as the natural space in which we look for the solutions of a differential equation; just to mention  a few of those readily seen as being natural:
\begin{itemize}
\item The Sobolev space $\,\mathscr W^{1,2}(\Omega)\,$ for the  Laplacian.
\item
The Sobolev space $\,\mathscr W^{1,p}(\Omega)\,$ for the $\,p$-harmonic operator \\$\, \textnormal{div}\, \big(|\nabla |^{p-2} \nabla  \big) = 0 \,$,
\item
The  space  $\,\mathscr W^{1,n}(\mathbb X ,\,\mathbb Y)\,$ for quasiconformal mappings $\, f\,:\,\mathbb X \onto \mathbb Y\,$ between $\,n$-manifolds, in which the major player is the Jacobian determinant $\, J(x,f)\,\textnormal{d}x = f^\sharp(\textnormal{d}y)\,$  -pullback of the volume form in $\,\mathbb Y .$
\end{itemize}
There is no genuine distinction between linear and nonlinear differential operators. Indeed,  once we depart from their natural domain of definition the application of singular integrals in the extended settings such as:
 $\,\mathscr L^p(\mathbb R^n) \rightarrow   \mathscr L^p(\mathbb R^n)\,$, \; $\,\mathscr L^1(\mathbb R^n) \rightarrow   \mathscr L^1_{\textnormal{weak}}(\mathbb R^n)\,$, \; $\,\mathscr H^1(\mathbb R^n) \rightarrow   \mathscr L^1(\mathbb R^n)\,$
and $\,\mathscr L^\infty(\mathbb R^n) \rightarrow   BMO(\mathbb R^n)\,$, becomes equally pressing in both cases.  But first we need some definitions.

 \section{Maximal Operators}
 Maximal inequalities, traditionally discussed in the entire space $\mathbb R^n\,$,  can actually be considered for functions  $\,f \in \mathscr L^1_{\textnormal{loc}} (\Omega)\,$ on any open set $\,\Omega \subset\mathbb R^n\,$:
 \begin{itemize}
 \item \textit{Hardy-Littlewood maximal function} \cite{HL} (1930)\,,
 $$\,\mathbf M f(x) = \sup \,\Big\{\;\dashint _B | f(y)|\,\textnormal{d}y\,;\;\;B = B(x, r) \subset \Omega \;\Big\}\, , \qquad \dashint_B = \frac{1}{|B|} \int_B $$

 \item \textit{Fefferman's sharp operator} \cite{F} \,(1971)\;,
   $$\,\;\;\;\;\;\;\mathbf M^\sharp f(x) = \sup_B \,\Big\{\;\dashint _B \big| f(y) \,- \, f_B\,\big|\,\textnormal{d}y\,;\;\;B = B(x, r) \subset \Omega \;\Big\} $$
 \item \textit{Spherical operator of E. Stein }\; \cite {St1}\;,
  $$\,\mathbf S f(x) = \sup_{\partial B} \,\Big\{\;\dashint _{\partial B} | f(y)|\,\textnormal{d}y\,;\;\;\;B = B(x, r) \subset \Omega \;\Big\}$$

  \begin{theorem}[Three Fundamental Maximal Inequalities]\label{maxIneq}$\;$\\

     For every  $\,f \in \mathscr L^ q(\mathbb R^n)\,$ with $\, 1 < q < \infty\,$, we have  \,\footnote{\,Hereafter the notation $\,\preccurlyeq\,$ refers to inequalities with the so-called \textit{implied constants} in the right hand side, which vary from line to line. Their precise values are readily perceived from the context. We shall indulge in this harmless convention for aesthetic reasons.}
  \begin{equation}\label{1}
  \|\, f\,\|_{\mathscr L^q(\mathbb R^n)} \preccurlyeq \; (q-1)\cdot\|\, \mathbf M f\,\|_{\mathscr L^q(\mathbb R^n)} \preccurlyeq \|\, f\,\|_{\mathscr L^q(\mathbb R^n)}\,.
  \end{equation}\\
   If $\,f \in \mathscr L^1(\mathbb R^n) \cap \mathscr L^q(\mathbb R^n)\,$ and $\, 1 < q < \infty\,$, then
  \begin{equation}\label{2}
  \|\, \mathbf M f\,\|_{\mathscr L^q(\mathbb R^n)} \preccurlyeq        \|\, \mathbf M^\sharp  f\,\|_{\mathscr L^q(\mathbb R^n)} \preccurlyeq  \|\, \mathbf M f\,\|_{\mathscr L^q(\mathbb R^n)}
  \end{equation}
     In both inequalities (\ref{1}) and  (\ref{2})  the implied constants stay bounded as $\, q \,$ approaches 1.\\

   \;If  $\,f \in \mathscr L^s(\mathbb R^n)\,$ with  $\, s > \frac{n}{n-1}\,$, then
  \begin{equation}\label{3}
  \|\, \mathbf S f\,\|_{\mathscr L^s(\mathbb R^n)} \preccurlyeq        \|\,  f\,\|_{\mathscr L^s(\mathbb R^n)}
  \end{equation}
  This time the implied constant blows up as $\, s\, $ approaches $\, \frac{n}{n-1}\,$.

  \end{theorem}

\end{itemize}
  Our discussion of the Hardy space $\,\mathscr H^1(\Omega)\,$ becomes somewhat simpler if we confine ourselves to a rotationally invariant approximation to the identity.   Thus we choose and fix a function $\,\Phi \in \mathscr C^\infty_\circ[0, 1 )\,$  such that $\,\int_{\mathbb R^n} \Phi (|x|) \,\textnormal{d}x \,= 1\,$, and set  $\, \Phi_t(x) =  t^{-n} \Phi\big(\frac{|x|}{t}\big)\,$.
 Given any $\,F\in \mathscr L^1_{\textnormal{loc}}(\Omega)\,$, we smooth it as
 $$
  \big(F \ast \Phi_t\big) (x) =  \int_\Omega \Phi_t(x-y) \,F(y)\,\textnormal{d}y\,\,,\;\;\;\;\;0 < t < \textnormal{dist} (x, \partial \Omega)
  $$
  A maximal operator that accounts for cancellations of positive and negative terms is now given by

\begin{itemize}
\item  $ \mathbf{\mathcal M}F (x) = \sup\big\{\;\big|F\ast \Phi_t \big|(x)\; : \;0 < t < \textnormal{dist} (x, \partial \Omega)\, \big\}
  $
 \item
\textit{The Hardy Space $\,\mathscr H^1(\Omega)\,$} consists of functions $\,F \in \mathscr L^1(\Omega)\,$ such that
 $$
 \big{\|} F \big{\|}_{\mathscr H^1(\Omega)}\;\deff  \; \big{\|} \mathbf{\mathcal M}F \big{\|}_{\mathscr L^1(\Omega)}\;< \; \infty
 $$
 \end{itemize}

 The above maximal operators, brilliantly developed by C. Fefferman and E. Stein, not only gave birth to a new discipline to effectively handle  singular integrals: they also provided Geometric Analysts, like me, with the means of solving demanding problems in Geometric Function Theory (GFT) and nonlinear PDEs. Maximal inequalities saved us from laborious computation once used in a clever, sometimes artistic, way. Let us take on stage, as  the first sketch, the Jacobian determinant  $\,F = J(x,f) = \textnormal{det}\, Df(x)\,$ of the differential matrix $\,Df(x) \in \mathbb R^{n\times n}\,$ of a  mapping $\,f : \Omega \rightarrow \mathbb R^n\,$ \cite{IO}.

 \section*{Sketch I,\;\; Let Hardy\,$\,\&\,$\,Littlewood meet Fefferman\,$\,\&\,$\,Stein,  \\ $\,J(x,f) \in \mathscr H^1(\Omega)\,$ }
  It is quite easily seen that for the  rotationally invariant approximation of unity, we have
 \begin{equation}\label{convolution}
 \big|F\ast \Phi_t \big|(x) \leqslant  \frac{C_\Phi}{t^{n+1}} \int_0^t \Big| \int_{B(x, r)} F(y)\,\textnormal{d}y \Big|\,\textnormal{d}r\;,\;\;\;\;\;\;\;C_\Phi = \big{\|} \Phi' \big{\|}_{\mathscr L^\infty[0, 1)}
 \end{equation}
 It should be noted that the absolute value is administered only upon integrating $\,F\,$  over the ball $\, B(x, r)\,$. Such an observation, though elementary,  is vital when dealing with  \textit{null-Lagrangians}. Null-Lagrangians, like Jacobians,  are  the  nonlinear differential expressions whose integral mean over any subdomain reduces to the boundary integral, basically due to cancellation of second order partial derivatives when integrating by parts.

 \begin{theorem} \label{Jacobian in Hardy space} \cite{IO} Suppose  $\,f \in \mathscr W^{1, n-1}_{\textnormal{loc}}(\Omega ,\, \mathbb R^n )\,$ and  the matrix of cofactors $\,D^\sharp f \in \mathscr L^{\frac{n}{n-1}}(\Omega, \mathbb R^{n\times n})\,$. Then the Jacobian determinant of $\,f\,$ lies in the Hardy space $\,\mathscr H^1(\Omega)\,$. Furthermore,
 \begin{equation}\label{sharp}
 \big{\|}\, \textnormal{det}\,Df \,\big{\|}_{\mathscr H^1(\Omega)}\;\preccurlyeq\; \int_\Omega |\,D^\sharp f(x)\, |^{\frac{n}{n-1}}\,\textnormal{d}x
 \end{equation}
 \end{theorem}
 \begin{proof} We sketch the proof  with emphasis  on  the spherical maximal function that comes into play.
 Let us commence with an isoperimetric type inequality

 \begin{equation}\label{isoperimetric}
 \Big|\; \dashint_{B(x,r)} F(y)\; \textnormal{d}y \;\Big|\;\preccurlyeq\; \Big(\;\dashint_{S(x,r)} \big|D^\sharp f (y)  \big| \textnormal{d}y \;\Big )^{\frac{n}{n-1}}\;,\;\;\;\;F(y)\,=\,J(y, f)
 \end{equation}
  Thus, in particular, $\, |\,F\,| \leqslant \,|D^\sharp f\,|^{\frac{n}{n-1}} \in \mathscr L^1(\Omega)\,$.
 A skillful reader with patience  can find it in Federer's  Book, see also \cite{MQY} for a legible proof.

  Trivially, we have a pointwise inequality $\,\mathbf{M}F \leqslant \mathbf{M}\big(|D^\sharp f|^{\frac{n}{n-1}} \big)\,$, but it does not yield the desired $\,\mathscr L^1\,$-integrability of $\,\mathbf{M}F \,$.  Combining  (\ref{convolution}) and (\ref{isoperimetric}) yields a better  inequality  $\,\mathbf{\mathcal M}F \preccurlyeq [ \textbf{S}(D^\sharp f) ]^{\frac{n}{n-1}}\,$. Unfortunately, we still find ourselves in a borderline of the $\,\mathscr L^p\,$-theory of the spherical maximal operator \cite{B,St1}. Neither Hardy-Littlewood nor Fefferman-Stein would ensure us that  $\,\mathbf{\mathcal M}F\in \mathscr L^1(\Omega)\,$. But together they actually come to the rescue. To this effect we introduce a one-parameter family $\,\{\mathfrak M_s\}_{1\leqslant s \leqslant \infty}\,$ of maximal operators,
 \begin{equation}
 \begin{split}
 \mathfrak M_s F(x) = &\underset{0 < t < \textnormal{dist}( x ,\, \partial \Omega)}{\textnormal{\Large{sup}}} \Big{[} \;\frac{n}{t^n} \int_0^t r^{n-1} \Big( \,\dashint _{S(x, r )} |F(y)|\,\textnormal{d}y \,\Big) ^s\,\textnormal{d} r\;\Big ]^{\frac{1}{s}}\\ & \;
 \end{split}
 \end{equation}
 The observant reader may wish to note that this family interpolates between $\,\mathbf{M} = \mathfrak M_1 \,$ and $\,\mathbf{S} = \mathfrak M_\infty\,$.
 \begin{theorem}
 The sublinear operator $\,\mathfrak M_s \,: \mathscr L^p(\Omega) \rightarrow \mathscr L^p(\Omega)\,$ is bounded for all exponents $\,p > \frac{n}{n-1 +\frac{1}{s}}\,$; thus for $\, p = \frac{n}{n-1}\,$ when $\,s \neq  \infty\,$.
 \end{theorem}
 Having this result in mind we now complete the proof of (\ref{sharp}), by using another straightforward consequence of  (\ref{convolution}) and (\ref{isoperimetric}),
 $$
 \big |\,F \ast \Phi_t \,\big | (x) \preccurlyeq \Big (\mathfrak M_s | D^\sharp f\,| \Big)^{\frac{n}{n-1}} \,\,,\;\;\;\textnormal{where, incidentally or not,}\;\; s = \frac{n}{n-1}\,.
  $$
 Thus $\,\mathbf{\mathcal M}F \preccurlyeq  \Big (\mathfrak M_s | D^\sharp f\,| \Big)^{\frac{n}{n-1}}\,$. In conclusion,
  $$
  \| \,F\,\|_{\mathscr H^1(\Omega)} \;=\;\| \,\mathbf{\mathcal M}F\,\|_{\mathscr L^1(\Omega)} \; \preccurlyeq  \int_\Omega |D^\sharp f |^{\frac{n}{n-1}}
  $$
  \end{proof} \textit{Exploring the $\,\textnormal{BMO}-\mathscr H^1\,$ duality.}  Two bonus results can readily be deduced from Theorem \ref{Jacobian in Hardy space}.
First, since $\,\textnormal{BMO}(\mathbb R^n)\,$  is the dual space to $\,\mathscr H^1(\mathbb R^n)\,$ \cite{F}, we obtain for every $\,\varphi \in \textnormal{BMO}(\mathbb R^n)\,$
 \begin{equation}
 \int_{\mathbb R^n} \varphi(x) \,J(x,f)\,\textnormal{d}x\;\preccurlyeq \,\| \varphi \|_{_{\textnormal{BMO}}}\,\int_{\mathbb R^n} \big|  D^\sharp f(x)\big|^{\frac{n}{n-1}}\,\textnormal{d}x \;,
 \end{equation}
  provided  $\,f \in \mathscr W^{1, n-1}(\mathbb R^n ,\, \mathbb R^n\,)\,$ and  $\,  |D^\sharp f | \in \mathscr L^{\frac{n}{n-1}}(\mathbb R^n)\,$.\\

 Second, since $\,\mathscr H^1(\mathbb R^n)\,$ is the dual of $\,\textnormal{VMO}(\mathbb R^n)\,$, we infer compactness  of the Jacobian determinants in the weak star topology of $\,\mathscr H^1(\mathbb R^n)\,$,
 \begin{equation}\nonumber
 \lim_{k \rightarrow \infty}\,\int_{\mathbb R^n} \varphi(x) \,J(x,f_k)\,\textnormal{d}x \;  = \;        \int_{\mathbb R^n} \varphi(x) \,J(x,f)\,\textnormal{d}x\;,\;\;\;\;\;\varphi \in \textnormal{VMO}(\mathbb R^n),
  \end{equation}
whenever $\,f_k\rightharpoonup f\,$ weakly in $\,\mathscr W^{1, n-1}(\mathbb R^n, \mathbb R^n )\,$ and $\,|D^\sharp f_k|\,$ stay bounded in $\mathscr L^{\frac{n}{n-1}}(\mathbb R^n)\,$. Be cautious, this fails for $\,\varphi(x) = \log |x|\,$.\\

\textit{Other implications.} Assuming that $\, f \in \mathscr W^{1,n}(\mathbb R^n, \mathbb R^n)\,$ we see that  Theorem \ref{Jacobian in Hardy space} covers the popular result of \cite{CLMS} on $\,\mathscr H^1$ -regularity of the Jacobians,  because $\, |D^\sharp f|^{\frac{n}{n-1}} \,\leqslant |Df|^n \in \mathscr L^1(\mathbb R^n)\,$. It also covers a very useful result by S. M\"{u}ller \cite{M1}  on local $\,\mathscr L \log \mathscr L\,$ - integrability of a nonnegative Jacobian. In fact we have, for every pair of concentric balls $\, B \subset 2 B \subset \mathbb R^n\,$,
\begin{equation}\label{LlogL}
\int _B F\cdot\log \Big(e + \frac{F}{F_B}  \Big)\; \preccurlyeq \;\big{\|} F \big{\|}_{\mathscr H^1(2B)}\,\preccurlyeq \int _{2B} | Df |^n \;,\; \;\;\; F = J(x, f) \geqslant 0
\end{equation}
Out of curiosity, the left hand side represents a norm in the Zygmund space $\,\mathscr L \log \mathscr L (B)\,$; the triangle inequality holds. Many further inequalities are to be found in \cite{IMb}. One of the central problems in GFT is to determine minimal regularity of a Sobolev map under which the nonnegative Jacobian is locally integrable. In fact  \cite{GIOV,IS1}, we have somewhat dual estimates below the natural domain of definition of the Jacobian function,
$$
\int _B J(x, f ) \textnormal{d}x  \,\preccurlyeq \int_{2B} | D^\sharp f |^{\frac{n}{n-1}} \;\textnormal{\large{log}}^{-1} \Big( \textnormal{\large e }+  \frac{|\,D^\sharp f |}{|\,D^\sharp f |_{2B}}\Big)      \preccurlyeq  \int_{2B}\; \frac{| D f(x) |^n \;\textnormal{d}x}{\textnormal{\large{log}} \Big( \,\textnormal{\large e }+  \frac{|\,D f(x) \,|}{|\,D f |_{2B}}\,\Big)}
$$
The true value of these estimates goes beyond theoretical interest; they play a significant role in establishing the existence of energy-minimal deformations in the theory of $n$-dimensional hyperelasticity.

\section*{Sketch II, \;\;The $\,p\,$ -harmonic transform, \\ a play with the sharp maximal operator }
We consider the nonhomogeneous  $\,p$ -harmonic equation, a prototype of many nonlinear PDEs,
\begin{equation}\label{p-harmEquation}
\, \textnormal{div}\, \big(|\nabla  u|^{p-2} \nabla u \big) = \textnormal{div}\,|\mathfrak f |^{p-2}\mathfrak f\,\;,\;\;\; 1<p< \infty \,
\end{equation}
The operator that carries a given vector field $\,\mathfrak f \in \mathscr L^p(\mathbb R^n\,, \mathbb R^n)\,$ into the gradient of the (unique) solution $\,u \in \mathscr W^{1,p}(\mathbb R^n)\,$ will be called $\,p\,$ -\textit{Harmonic Transform}, denoted by
$$
\mathbf{R}_p \,: \mathscr L^p(\mathbb R^n ,\,\mathbb R^n) \rightarrow \mathscr L^p(\mathbb R^n ,\,\mathbb R^n)\;,\;\;\;\;\; \mathbf{R}_p \,\mathfrak f \,\deff \,\nabla u
$$
The linear operator $\,\mathbf R_2\, = \, - \big[\,R_{ij} \,\big]_{1\leqslant i, j \leqslant n}\,:\, \mathscr L^s(\mathbb R^n, \mathbb R^n) \rightarrow \mathscr L^s(\mathbb R^n, \mathbb R^n)\,$ is  a matrix of second order Riesz transforms, $\,R_{ij} = R_i\circ R_j\,$. This is a device for the $\,\mathscr L^2$ -projection of a vector field onto the gradient and divergence-free components, known as \textit{Hodge decomposition}:
\begin{equation}\label{Hodge decomposition}
\mathfrak f = \nabla \varphi \, + \,\frak h \;=\; \textbf{R}_2\mathfrak f \,+\,\textbf{T}\mathfrak f \;,\;\;\;\; \textbf{T }= \textbf{I} - \textbf{R}_2\,:\, \mathscr L^s(\mathbb R^n, \mathbb R^n) \rightarrow \mathscr L^s(\mathbb R^n, \mathbb R^n)\nonumber
\end{equation}
Note that $\,\textbf{T}\,$ vanishes on gradient fields. The idea in the sequel is to use $\,\varphi\,$ as a test function for a divergence type nonlinear differential expressions.\\
While the Sobolev space $\,\mathscr W^{1,p}(\mathbb R^n)\,$, together with the given vector field $\,\mathfrak f \in \mathscr L^p(\mathbb R^n\,, \mathbb R^n)\,$, is considered the natural setting for the equation (\ref{p-harmEquation}), we shall depart from it and move into the realm of exponents $\, s \geqslant p\,$. We shall show that the $\,p\,$-harmonic transform is bounded in $\,\mathscr L^s(\mathbb R^n ,\, \mathbb R^n)\,$ in the sense of the following
\begin{theorem} If $\,\mathfrak f\,$ belongs to $\, \mathscr L^p(\mathbb R^n ,\, \mathbb R^n) \, \cap \,\mathscr L^s(\mathbb R^n ,\, \mathbb R^n)\,$ then so does $\, \textbf{R}_p\,\mathfrak f \,$. Moreover, we have the uniform bound,
\begin{equation}\label{bound of R}
\, \|\,\textbf{R}_p\,\mathfrak f\,\|_s \preccurlyeq \|\,\frak f\,\|_s \,, \;\;\;\; s \geqslant p
\end{equation}

\end{theorem}

\begin{proof} We only sketch the proof of the uniform estimate (\ref{bound of R} ), and stick to the case $\, p \geqslant 2\,$  for simplicity. A  laborious proof of $\,\mathscr L^s\,$ -integrability of $\,\textbf{R}_p\mathfrak f\,$, based on Gehring's Lemma on reverse H\"{o}lder inequalities \cite{Ge},  can be found in \cite{I1}.
Choose and fix a ball $\, B \subset \mathbb R^n\,$. The weak form of equation (\ref{p-harmEquation}) reads as,
\begin{equation}\label{pharm}
\int_B \langle\;|\nabla u |^{p-2} \nabla u\; \big |\; \nabla\varphi\, \rangle \;=\; \int_B \langle\;|\mathfrak f|^{p-2} \mathfrak f\; \big |\; \nabla\varphi\, \rangle\;,\;\;\;\textnormal{for}\;\;\;\varphi \in \mathscr W^{1,p}_\circ(B)
\end{equation}
Let  $\,v \in u + \mathscr W^{1,p}_\circ(B)\,$ be a (unique) function that agrees with $\,u\,$ on $\,\partial B\,$ and has least $\,p\,$-harmonic energy $\, \mathscr E_p[v] = \int_B |\nabla v |^p\,$ . Thus $\,v\,$ is a $\,p\,$-harmonic function, meaning that

\begin{equation}\label{p-harm}
\int_B \langle\;|\nabla v |^{p-2} \nabla v\; \big |\; \nabla\varphi\, \rangle \; = 0 \;\;\; \textnormal{and}\,\;\; \dashint _B |\nabla v |^p \leqslant  \dashint _B |\nabla u |^p
\end{equation}
We subtract this integral from the left hand side of (\ref{pharm}) and test the resulting equation with $\, \varphi = u - v\,$. From this it is straightforward to derive basic local estimates
\begin{equation}\label{locest}
\dashint_B |\nabla u - \nabla v |^p \;\;\preccurlyeq \;\; \dashint_B |\,\mathfrak f|^p
\end{equation}
We aim to replace $\,\nabla v\,$ by a constant.  For this we recall the $\,\mathscr C^{1, \alpha}\,$ -regularity of $\,p\,$-harmonic functions, $\, 0< \alpha = \alpha(n, p) \leqslant 1\,$.  Precisely, for every $\,0 < \tau \leqslant 1\,$,
\begin{equation}\label{locest1}
\dashint_{\tau B} \big |\,\nabla v \,- \,(\nabla v )_{\tau B} \;\big |^p \;\;\preccurlyeq \;\; \tau ^{\alpha \,p}\dashint_B |\nabla v|^p \;\leqslant \,\tau ^{\alpha \,p}\dashint_B |\nabla u|^p
\end{equation}
where the implied constants in the inequalities  $\,\preccurlyeq\,$ depend only on $\,n\, $ and $\,p\,$. On the other hand (\ref{locest}) yields
\begin{equation}\label{locest2}
\dashint_{\tau B} \big |\,\nabla u \,- \,\nabla v  \;\big |^p \;\;\preccurlyeq \;\; \tau ^{-n}\dashint_B |\,\mathfrak f\,|^p
\end{equation}
whence it is readily inferred that
\begin{equation}\label{locest3}
\dashint_{\tau B} \big |\,\nabla u \,- \,(\nabla u)_{\tau B}  \;\big |^p \;\;\preccurlyeq \;\; \tau ^{-n}\dashint_B |\,\mathfrak f\,|^p\;\;+\;\;\tau ^{\alpha \,p}\dashint_B |\nabla u|^p
\end{equation}

From now on there are two ways to obtain the estimate (\ref{bound of R}), both via a pointwise inequality between maximal functions. In the first approach, we apply H\"{o}lder's inequality to the left hand side of (\ref{locest3}) and then take supremum over all balls centered at a given point,

$$
\big|\,\mathbf M^\sharp \nabla u\,\big|^p \preccurlyeq \tau ^{-n}\mathbf M |\,\mathfrak f\,|^p\;\;+\;\;\tau ^{\alpha \,p}\mathbf M|\nabla u|^p\;,\;\;\;\;\textnormal{pointwise in }\;\mathbb R^n
$$
We raise this to the power $\,\frac{s}{p} > 1\,$ and, with the aid of maximal inequalities, obtain

$$
\big{\|}\,\nabla u\,\big{\|}^p_{\mathscr L^s(\mathbb R^n)} \; \preccurlyeq \;\tau ^{-n}\big{\|}\,\mathfrak f\,\big{\|}^p_{\mathscr L^s(\mathbb R^n)}\;\;+\;\;\tau ^{\alpha \,p}\big{\|}\,\nabla u\,\big{\|}^p_{\mathscr L^s(\mathbb R^n)}\,,
$$
where the implied constant depends on $\,n, p\,$ and $\,s\,$, but not on the parameter $\,\tau\,$. The observant reader may be concerned that this constant blows up as $\, s\,$ approaches $\,p\,$. But still we can chose $\,\tau\,$ small enough so that the last term in the right hand side will be absorbed by the left hand side, establishing the desired estimate (\ref{bound of R}).\\
  In the second approach, to avoid the undue anomaly  near the natural Sobolev space $\,\mathscr W^{1, p}(\mathbb R^n)\,$, we relinquish the idea of using the sharp maximal inequality (\ref{2}). Instead, we appeal to the full force of Hardy-Littlewood maximal inequalities near $\,\mathscr L^1(\mathbb R^n)\,$. For this purpose, we rewrite  (\ref{locest3}) as

\begin{equation}\label{locest4}
\dashint_{\tau B} \big |\,\nabla u  \;\big |^p \;\;\preccurlyeq  \;\;\tau ^{-n}\dashint_B |\,\mathfrak f\,|^p\;+\;\big|\,(\nabla u)_{\tau B}\;\big|^p \;\;+\;\;\tau ^{\alpha \,p}\dashint_B |\nabla u|^p\;,
\end{equation}
Taking supremum over the balls centered at a given point we capture a pointwise inequality for maximal functions
\begin{equation}\label{locest4}
\mathbf M \,|\nabla u  \; |^p \;\;\preccurlyeq \;\,\tau ^{-n}\mathbf M\,|\,\mathfrak f\,|^p\;\;+\;\;|\,\mathbf M\,\nabla u\;|^p \;\;+ \;\;\tau ^{\alpha \,p}\mathbf M \,|\nabla u|^p\;,
\end{equation}
We now eliminate the operator $\,\mathbf M\,$ by computing the $\,\mathscr L^{q}(\mathbb R^n)\,$ -norm, $\,q = \frac{s}{p} \approx 1\,$,  of both sides.  Maximal inequalities  (\ref{1}) yield,
$$
\big {\|}\,\nabla u  \;\big {\|}^p_{\mathscr L^s(\mathbb R^n) }\preccurlyeq \,\tau ^{-n}\big {\|}\,\mathfrak f  \;\big {\|}^p_{\mathscr L^s(\mathbb R^n) }\;+\;(s-p)\,\big {\|}\,\nabla u  \;\big {\|}^p_{\mathscr L^s(\mathbb R^n) } \;+ \;\tau ^{\alpha \,p}\big {\|}\,\nabla u  \;\big {\|}^p_{\mathscr L^s(\mathbb R^n) }
$$
This time we are not troubled with the exponent $\,s\,$ approaching $\,p $ ; the implied constant remains bounded. Choose $\, s = s(n, p) > p\,$ close enough to $\,p\,$ and $\,\tau\,$ sufficiently small so that the last two terms will be absorbed by the left hand side. We obtain uniform bounds (\ref{bound of R})  for $\,\mathbf R_p\,$ near its natural domain of definition.
\end{proof}

 \begin{remark}
 The unplanned bonus from the asymptotically precise Hardy-Littlewood maximal inequalities  (\ref{1}) gives the above proof its beauty, doesn't it?
\end{remark}

\section*{Sketch III,\;\;The splendor of commutators}
 In any preliminary analysis of a differential equation, linear or nonlinear, one often encounters undesirable higher order terms which eventually cancel out. The instruments for rigorous performance of such an analysis are the commutators of a singular integral $\,\mathbf T : \mathscr L^s(\mathbb R^n) \rightarrow  \mathscr L^s(\mathbb R^n)\,,\;1 < s < \infty \,,\;$ with suitable nonlinear algebraic operations on the gradient of the solution. Let us look briefly at three commutators,  together with their underlying estimates.
 \begin{itemize}
 \item The linear commutator of Coifman-Rochberg-Weiss \cite{CRW}
 $$
 \big{\|}\, \mathbf T (\lambda f ) \;-\; \lambda (\mathbf T f ) \,\big{\|} _ {\mathscr L^s(\mathbb R^n)}\;   \preccurlyeq \;\big{\|}\, \lambda\,\big{\|} _{BMO(\mathbb R^n)} \, \big{\|} \,f \,\big{\|} _{\mathscr L^s(\mathbb R^n)}
 $$
 \item The Rochberg-Weiss commutator \cite{RW}
  $$
 \big{\|}\, \mathbf T ( f \log |f| ) \;-\; (\mathbf T f )\log |\mathbf T f| \,\big{\|} _ {\mathscr L^s(\mathbb R^n)}\;   \preccurlyeq \; \, \big{\|} \,f \,\big{\|} _{\mathscr L^s(\mathbb R^n)}
 $$
 \item The commutator of $\,\mathbf T\,$ and a power type operation \cite{IS2}
 $$
 \big{\|}\, \mathbf T ( |f|^{\pm \,\varepsilon} f ) \;-\; |\mathbf T f|^{\pm \,\varepsilon} (\mathbf T f ) \,\big{\|} _ {\mathscr L^s(\mathbb R^n)}\;   \preccurlyeq \; \, |\varepsilon|\cdot \big{\|} \,|f|^{1\pm \,\varepsilon} \,\big{\|} _{\mathscr L^s(\mathbb R^n)}\;,\;\;\; 0\leqslant \varepsilon < \,1- \frac{1}{s}
 $$
 \end{itemize}
The proof of this latter estimate captures the ideas of the complex method of interpolation originated in the celebrated work by G. O. Thorin \cite{T}. Actually, it yields the estimate of the Rochberg-Weiss commutator through the L'H\^{o}pital's rule, which in turn gives us M\"{u}ller's  $\,\mathscr L \log \mathscr L $ -integrability of nonnegative Jacobians in a stylish way \cite{IMb}.

Although the linear commutator of Coifman-Rochberg-Weiss has been known for a long time, and numerous deep studies have been devoted to it, its usefulness in solving PDEs still remains magical. For example, good estimates of the $\,p$ -norms of the tensor products of the Riesz transforms combined with the Fredholm index theory (via compactness of the Coifman-Rochberg-Weiss commutators) are elegant tools in elliptic PDEs with VMO coefficients \cite{IS3}; there is no need to go again and again through the foundational details of singular integrals.\\

I have saved the best for last:
\subsection*{\textit{Very weak solutions of nonlinear PDEs}} These are the solutions weaker than those in the natural domain of definition. The chief difficulty is to launch some estimates in order to take the very weak solution off the ground. Let us return to the weak formulation of  the $\,p$ -harmonic equation (\ref{pharm}) in which $\,\mathfrak f \in \mathscr L^{p-\varepsilon}(\mathbb R^n ,\mathbb R^n)\,$, so we must look, naturally, for the solution   $\, u \in \mathscr W^{1,\, p-\varepsilon}(\mathbb R^n,\,\mathbb R^n)\,$. The legitimate test function $\,\varphi\,$ must lay in $\, \mathscr W^{1,\,\frac{p-\varepsilon}{1-\varepsilon}} (\mathbb R^n)\,$. We steal it from the Hodge decomposition
$$
|\nabla u |^{-\varepsilon}  \nabla u \; = \nabla \varphi \,+\,\mathfrak h\,,\;\;\; \textnormal{where}\;\; \mathfrak h  =  \mathbf T(|\nabla u |^{-\varepsilon}  \nabla u )\;,
$$
to obtain
\begin{equation}\label{xx}
\int_{\mathbb R^n} |\nabla u |^{p -\varepsilon}\preccurlyeq \int_{\mathbb R^n} |\mathfrak f |^{p -\varepsilon}\;+\; \int_{\mathbb R^n} |\mathfrak h |^{\frac{p -\varepsilon}{1 -\varepsilon}}
\end{equation}

Since the operator $\,\mathbf T\,$ vanishes on gradient fields we can write $\,\mathfrak h\,$ as a commutator of $\,\mathbf T\,$ and the power function, $\,\mathfrak h  =  \mathbf T(|\nabla u |^{-\varepsilon}  \nabla u )\;-\;|\mathbf T\nabla u |^{-\varepsilon} ( \mathbf T \nabla u ) \,$. Our estimate for the power type commutator shows that
$$
 \int_{\mathbb R^n} |\mathfrak h |^{\frac{p -\varepsilon}{1 -\varepsilon}} \preccurlyeq\; \varepsilon\cdot \int_{\mathbb R^n} |\nabla u |^{p -\varepsilon}
$$
Consequently, this term can be sucked up  by the left hand side, which results in the desired estimate of the $\,p$ -harmonic transform slightly below its natural domain of definition.
\begin{equation}\label{xxx}
\big{\|}  \textbf{R}_p \,\mathfrak f \big{\|}_{\mathscr L^{p -\varepsilon}(\mathbb R^n) }  \preccurlyeq   \big{\|}  \mathfrak f \big{\|} _{\mathscr L^{p -\varepsilon}(\mathbb R^n)}
\end{equation}
A study of very weak solutions of nonlinear PDEs is largely motivated by removability of singularities \cite{Ia}.

\section{Moral of the Story}
If one dark rainy night you find yourself in the midst of Whitney cubes, covering lemmas, Calder\'{o}n-Zygmund decomposition, etc., then you should remind yourself that instead you might cleverly apply singular integral stuff and thereby see the light.

\begin{quote}
\textit{``Every block of stone has a statue inside it and it is the task of the sculptor to discover it.''}
\end{quote}
\begin{flushright} - \textit{Michelangelo di Lodovico Buonarroti Simoni}
\end{flushright}

\begin{quote}
\textit{``No profit grows where is no pleasure taken; in brief, sir, study what you most affect.''}
\end{quote}
\begin{flushright} - \textit{William Shakespeare \footnote{\;An English teacher assigned a student to read some Shakespeare, and a week later he asked: ``How did you like it?'' The student answered: ``Well, nothing special; just a collection of quotations.''}}
\end{flushright}
However,
\begin{quote}
\textit{``In order for something to remain beautiful, it must stay long enough to be noticed and enjoyed, never so long as to outstay its welcome.''}
 \end{quote}
 \begin{flushright}  -\textit{Fisher, Philip} \\``The Rainbow and Cartesian Wonder.''
 \end{flushright}

Finally, as Thomas Edison said, \begin{center} \textit{``No sooner does a fellow succeed in making a good thing, than some other fellows pop up and tell you they did it years ago.''}.   \footnote{\;Once a  scholar loudly informed the speaker of a plenary lecture ``Sir, your result follows from my theorem, I proved it years ago.'' The speaker answered: ``Yes it does follow, obviously. However, I am not sure if your theorem follows from my result.''}\end{center}

\bibliographystyle{amsplain}

\begin{thebibliography}{99}







\bibitem{B}
J. Bourgain \textit{Estimations de certaines fonctions maximales [Estimates of some maximal operators]}, C.R. Acad. Sci. Paris St. I Math.  \textbf{301} (10) (1985), 499--502.




\bibitem{CLMS}
R. Coifman, P.-L. Lions, Y. Meyer and S. Semmes,  \textit{Compensated Compactness and Hardy Spaces},  J. Math. Pures Appl. \textbf{72} (9) (1993), 247--286.. Arch. Ration. Mech. Anal. {\bf 195}, no.~3 (2010), 899--921.

\bibitem{CRW}
R. Coifman, R. Rochberg and G.Weiss,  \textit{Factorization theorems for Hardy spaces in several variables}, Ann. of Math. (2) \textbf{103} (1976),  611--635.


\bibitem{F}
C. Fefferman, \textit{Characterizations of bounded mean oscillation}, Bull. Amer. Math. Soc. 77
(1971), 587-588.


\bibitem{FS}
C. Fefferman and E.M. Stein, \textit{$H^p\,$  spaces of several variables}, Acta Math. 129
(1972), 137–193.


\bibitem{Ge}
F. W.  Gehring,  \textit{The $L^{p}$-integrability of the partial derivatives of a quasiconformal mapping}, Acta Math. 130 (1973), 265--277.

\bibitem{GIOV}
F. Giannetti, T. Iwaniec, J. Onninen and A. Verde,  \textit{Estimates of Jacobians by Subdeterminants}, J. Geometric Analysis, vol. 12, no 2, (2002), 223--254.



\bibitem{HL}

G. H. Hardy and J. E. Littlewood, \textit{A maximal theorem with function-theoretic
applications}, Acta Math. 54 (1930) 81--116; also in Collected Papers of G. H.
Hardy, v. 2, Oxford Univ. Press, 1967, pp. 509--545.

\bibitem{I1}
T. Iwaniec,
\textit{Projections onto gradient fields and $\,\mathscr L^p$ -estimates for degenerate elliptic operators}, Studia Mathematica, T. LXXV. (1983), 293--312.



\bibitem{I}
T. Iwaniec,
\textit{An Essay on the Interpolation Theorem of J\'{o}zef Marcinkiewicz - Polish Patriot },  arXiv:1102.0167v1
\bibitem{Ia}
T. Iwaniec, \textit{$p$-harmonic tensors and quasiregular mappings}, Ann. of Math. (2) 136 (1992), no. 3, 589--624.

\bibitem{IMb}
T.  Iwaniec and G. Martin, \textit{Geometric function theory and non-linear analysis},  Oxford University Press, New York, 2001.



\bibitem{IO}
T.  Iwaniec and J. Onninen, \textit{$\mathscr H^1$ -estimates of Jacobians by subdeterminants},  Mathematische Annalen, 324, 341-358 (2002).

\bibitem{IS1}
T. Iwaniec,  and C. Sbordone, \textit{On the inegrability of the Jacobian under minimal hypothesis},  Arch. Rat. Mech. Anal. 119 (1992), 129--143.

\bibitem{IS2}
T. Iwaniec,  and C. Sbordone, \textit{Weak minima of variational integrals}, J.Reine Angew. Math. 454 (1994), 143--161.



\bibitem{IS3}
T. Iwaniec,  and C. Sbordone, \textit{Riesz transforms and elliptic PDEs with VMO coefficients}, Journal D'Analyse Math\'{e}matique, vol 74 (1998), 183--212.


\bibitem{M1}
J. Marcinkiewicz, \textit{Sur l'interpolation d'op\'{e}rations,}, Comptes Rendus des S\'{e}ances de l'Acad\'{e}mie des Sciences. S\'{e}rie A e B, 208 (1939),   1272--1273.




\bibitem{M1}
S. M\"{u}ller,   \textit{Higher integrability of determinants and weak convergence in $\,\mathscr L^1\,$} J.Reine Angew. Math. \textbf{412} (1990), 20--34.


\bibitem{MQY}
S. M\"{u}ller, T. Qi  and B.S. Yan,  \textit{On a new class of elastic deformations not allowing for cavitation} Ann. Inst. H. Poincar\'{e} Anal. Non Lineaire \textbf{11} (2) , 217-243 (1994).


\bibitem{RW}
R. Rochberg and G. Weiss, \textit{Derivatives of analytic families of Banach spaces},  Annals Math. 118 (1983) , 315--347.


\bibitem{St1}
E. M. Stein, \textit{Maximal functions: spherical means} Proc. Nat. Acad. Sci. U.S.A. 73 (1976), 2174–-2175

\bibitem{T}
G.O. Thorin, \textit{An extension of a convexity theorem due to M. Riesz,} Kungl. Fysiogr. S\"{a}llsk. i Lund F\"{o}rh. 8,  (1938), 166--170.




\end{thebibliography}

\end{document}